\documentclass[reqno]{amsart}

\usepackage[utf8]{inputenc} 
\usepackage{amssymb}
\usepackage{graphicx}
\usepackage{latexsym}
\usepackage{stmaryrd}
\usepackage{enumitem}
\usepackage[bookmarks,hyperfootnotes=false]{hyperref}
\usepackage{multirow}
\usepackage{xcolor}
\usepackage{caption}
\colorlet{darkishRed}{red!80!black}
\colorlet{darkishBlue}{blue!60!black}
\colorlet{darkishGreen}{green!60!black}
\usepackage{hyperref}
\hypersetup{
pdftitle = {Duality theorems for stars and combs},
pdfsubject = {Duality theorems for stars and combs},
pdfauthor = {Carl Bürger and Jan Kurkofka},
pdfkeywords = {},
draft = false,
bookmarksopen=true}

\usepackage{theoremref} 
\usepackage{tikz}
\usepackage{tikz-cd}
\usetikzlibrary{calc,through,intersections,arrows, trees, positioning, decorations.pathmorphing, cd}
\usepackage{relsize}
\usepackage{comment}
\usepackage{xifthen}
\usepackage{mathrsfs}
\usepackage{svg}
\usepackage{mathtools}
\usepackage{nccmath}
\usepackage{pifont}

\newcommand{\SCnumberCite}[1]
{%
\ifthenelse{\equal{#1}{1}}{\cite{StarComb1StarsAndCombs}}{}%
\ifthenelse{\equal{#1}{2}}{\cite{StarComb2TheDominatedComb}}{}%
\ifthenelse{\equal{#1}{3}}{\cite{StarComb3TheUndominatedComb}}{}%
\ifthenelse{\equal{#1}{4}}{\cite{StarComb4TheUndominatingStar}}{}%
}

\newcommand{\SCnumberHand}[2]
{%
\ifthenelse{\equal{#1}{#2}}{\ding{43}\,}{}%
}

\newcommand{\SCintroList}[2]
{%
    \ifthenelse{\equal{#1}{#2}}{(this paper)}{\SCnumberCite{#1}}%
}

\newcommand{\SCintroDetermined}[1]
{%
    \ifthenelse{\equal{#1}{1}}{In this paper, we determine}{In the first paper of this series, we determined}%
}

\newcounter{quotecount}

\makeatletter
\newcommand*{\addFileDependency}[1]{
  \typeout{(#1)}
  \@addtofilelist{#1}
  \IfFileExists{#1}{}{\typeout{No file #1.}}
}
\makeatother

\newcommand{\Abs}[1]{\partial_{\Omega} {#1}}

\newcommand{\rest}{\upharpoonright}

\renewcommand{\subset}{\subseteq}

\def\TFAD{Let $G$ be any connected graph and let $U\subset V(G)$ be any vertex set. Then the following assertions are complementary:}
\def\tfad{the following assertions are complementary:}
\newcommand{\at}{attached to }

\newcommand{ \N } { \mathbb{N} }

\newcommand{\dblue}[1]{\textcolor{darkishBlue}{#1}}

\newcommand{\dc}[1]{\lceil #1\rceil}
\newcommand{\uc}[1]{\lfloor #1\rfloor}
\newcommand{\guc}[1]{\lfloor\mkern-1.4\thinmuskip\lfloor #1\rfloor\mkern-1.4\thinmuskip\rfloor}

\newcommand{\dt}{T_{\textup{\textsc{dec}}}}

\makeatletter

\def\calCommandfactory#1{%
   \expandafter\def\csname c#1\endcsname{\mathcal{#1}}}
\def\frakCommandfactory#1{%
   \expandafter\def\csname frak#1\endcsname{\mathfrak{#1}}}

\newcounter{ctr}
\loop
  \stepcounter{ctr}
  \edef\X{\@Alph\c@ctr}
  \expandafter\calCommandfactory\X
  \expandafter\frakCommandfactory\X
  \edef\Y{\@alph\c@ctr}
  \expandafter\frakCommandfactory\Y
\ifnum\thectr<26
\repeat

\renewcommand{\cC}{\mathscr{C}}

\setenumerate{label={\normalfont (\roman*)},itemsep=0pt}

\def\lowfwd #1#2#3{{\mathop{\kern0pt #1}\limits^{\kern#2pt\raise.#3ex
\vbox to 0pt{\hbox{$\scriptscriptstyle\rightarrow$}\vss}}}}
\def\lowbkwd #1#2#3{{\mathop{\kern0pt #1}\limits^{\kern#2pt\raise.#3ex
\vbox to 0pt{\hbox{$\scriptscriptstyle\leftarrow$}\vss}}}}
\def\fwd #1#2{{\lowfwd{#1}{#2}{15}}}
\def\Sinf{S_{\aleph_0}}

\def\vE{{\hskip-1pt{\fwd{E}{3.5}}\hskip-1pt}}

\def\ve{\kern-1.5pt\lowfwd e{1.5}2\kern-1pt}
\def\ev{\kern-1pt\lowbkwd e{0.5}2\kern-1pt}
\def\vf{\kern-2pt\lowfwd f{2.5}2\kern-1pt}

\newtheorem{theorem}{Theorem}[section] 
\newtheorem{corollary}[theorem]{Corollary}
\newtheorem{lemma}[theorem]{Lemma}

\newtheorem{mainresult}{Theorem} 
\newtheorem{maincorollary}[mainresult]{Corollary}

\newtheorem{innercustomthm}{Theorem}

\newtheorem{innercustomcor}{Corollary}
\newenvironment{customthm}[1]
  {\renewcommand\theinnercustomthm{#1}\innercustomthm}
  {\endinnercustomthm}

\newenvironment{customcor}[1]
  {\renewcommand\theinnercustomcor{#1}\innercustomcor}
  {\endinnercustomcor}
  
\theoremstyle{definition}
\newtheorem{example}[theorem]{Example}

\theoremstyle{remark}

\newcommand{\kn}{knobbly}

\begin{document}

\title[Duality theorems for stars and combs III: Undominated combs]{Duality theorems for stars and combs\\
III: Undominated combs}

\author{Carl Bürger}
\author{Jan Kurkofka}
\address{University of Hamburg, Department of Mathematics, Bundesstraße 55 (Geomatikum), 20146 Hamburg, Germany}
\email{carl.buerger@uni-hamburg.de, jan.kurkofka@uni-hamburg.de}


\@namedef{subjclassname@2020}{\textup{2020} Mathematics Subject Classification}

\begin{abstract}
In a series of four papers we determine structures whose existence is dual, in the sense of complementary, to the existence of stars or combs.
Here, in the third paper of the series, we present duality theorems for a combination of stars and combs: undominated combs.
We describe their complementary structures in terms of rayless trees and of tree-decompositions.

Applications include a complete characterisation, in terms of normal spanning trees, of the graphs whose rays are dominated but which have no rayless spanning tree.
Only two such graphs had so far been constructed, by Seymour and Thomas~\cite{endfaithfulCounterexample} and by Thomassen~\cite{ThomassenEndfaithfulCounterexample}.
As a corollary, we show that graphs with a normal spanning tree have a rayless spanning tree if and only if all their rays are dominated.
\end{abstract}
\vspace*{-1.14cm} 
\maketitle

\vspace*{-.75cm}

\section{Introduction}

\noindent Two properties of infinite graphs are \emph{complementary} in a class of infinite graphs if they partition the class.
In a series of four papers we determine structures whose existence is complementary to the existence of two substructures that are particularly fundamental to the study of connectedness in infinite graphs: stars and combs.
See~\cite{StarComb1StarsAndCombs} for a comprehensive introduction, and a brief overview of results, for the entire series of four papers (\cite{StarComb1StarsAndCombs,StarComb2TheDominatedComb,StarComb4TheUndominatingStar} and this paper).

In the first paper~\cite{StarComb1StarsAndCombs} of this series we found structures whose existence is complementary to the existence of a star or a comb attached to a given set $U$ of vertices, and two types of these structures turned out to be relevant for both stars and combs: normal trees and tree-decompositions.
A \emph{comb} is the union of a ray $R$ (the comb's \emph{spine}) with infinitely many disjoint finite paths, possibly trivial, that have precisely their first vertex on~$R$. 
The last vertices of those paths are the \emph{teeth} of this comb.
Given a vertex set $U$, a \emph{comb attached to} $U$ is a comb with all its teeth in $U$, and a \emph{star attached to} $U$ is a subdivided infinite star with all its leaves in $U$.
Then the set of teeth is the \emph{attachment set} of the comb, and the set of leaves is the \emph{attachment set} of the star.
Given a graph $G$, a rooted tree $T\subset G$ is \emph{normal} in $G$ if the endvertices of every $T$-path in $G$ are comparable in the tree-order of $T$, cf.~\cite{DiestelBook5}.
For the definition of tree-decompositions see~\cite{DiestelBook5}.

As stars and combs can interact with each other, this is not the end of the story.
For example, a given vertex set $U$ might be connected in a graph $G$ by both a star and a comb, even with infinitely intersecting sets of leaves and teeth. 
To formalise this, let us say that a subdivided star $S$ \emph{dominates} a comb $C$ if infinitely many of the leaves of $S$ are also teeth of $C$.
A \emph{dominating star} in a graph $G$ then is a subdivided star $S\subset G$ that dominates some comb $C\subset G$; and a \emph{dominated comb} in $G$ is a comb $C\subset G$ that is dominated by some subdivided star $S\subset G$.
Thus, a comb $C\subset G$ is undominated in $G$ if it is not dominated in $G$.
Recall that a vertex $v$ of $G$ \emph{dominates} a ray $R\subset G$ if there is an infinite $v$--$(R-v)$ fan in~$G$, see~\cite{DiestelBook5}. 
A ray $R\subset G$ is \emph{dominated} if some vertex of $G$ dominates it.
Rays not dominated by any vertex of $G$ are \emph{undominated}.
Dominated combs are related to dominated rays in that a comb is dominated in $G$ if and only if its spine is dominated in $G$.

In the second paper~\cite{StarComb2TheDominatedComb} of our series we determined structures whose existence is complementary to the existence of dominating stars or dominated combs---again in terms of normal trees or tree-decompositions.

Here, in the third paper of the series, we determine structures whose existence is complementary to the existence of undominated combs.
A candidate for a normal tree that is complementary to an undominated comb in $G$ attached to a given set $U$ of vertices is a normal tree $T\subset G$ that contains $U$ and all whose rays are dominated in~$G$, for if $U=V(G)$ then $T$ is spanning and hence its (dominated) rooted rays are in a natural one-to-one correspondence to the ends of $G$.
Such normal trees $T$ are easily seen to be complementary structures for undominated combs whenever $G$ happens to contain some normal tree that contains~$U$.
But in general, normal trees $T\subset G$ containing $U$ all whose rays are dominated in $G$ are not complementary to undominated combs, because the absence of an undominated comb does not imply the existence of such a normal tree:
for example if $G$ is an uncountable complete graph and $U=V(G)$, then every normal tree in $G$ containing $U$ must be spanning but $G$ does not have any normal spanning tree.

As our first main result, we show that if $U$ is contained in any normal tree $T\subset G$, there is a more elementary structure that is complementary to undominated combs \at $U$ and which obstructs undominated combs \at $U$ immediately: a rayless tree containing~$U$.
Call a set $U\subseteq V(G)$ of vertices of a graph $G$ \emph{normally spanned} in $G$ if $U$ is contained in a tree $T\subset G$ that is normal in $G$. 
The graph $G$ is \emph{normally spanned} if $V(G)$ is normally spanned in $G$, i.e., if $G$ has a normal spanning tree. 
\begin{customthm}{\ref{thm: undominated comb duality}}
Let $G$ be any graph and let $U \subseteq V(G)$ be normally spanned in $G$. Then \tfad 
\begin{enumerate}
    \item $G$ contains an \dblue{undominated comb} \at $U$;
    \item there is a \dblue{rayless tree} $T\subset G$ that contains $U$.
\end{enumerate}
\end{customthm}
\noindent This extends results of Polat~\cite{polatFrenchIII,Polat90} and Širáň~\cite{siran}, who proved the case  $U=V(G)$ for countable~$G$: 
\emph{A countable connected graph has a rayless spanning tree if and only if all its rays are dominated.}

There are uncountable graphs $G$ for which this duality fails, even for $U=V(G)$.
By Theorem~\ref{thm: undominated comb duality}, such graphs $G$ cannot have a normal spanning tree.
There are two known constructions of such graphs, by Seymour and Thomas~\cite{endfaithfulCounterexample} and by Thomassen~\cite{ThomassenEndfaithfulCounterexample}.
Both these constructions are involved.

As a corollary of Theorem~\ref{thm: undominated comb duality} we obtain a full characterisation of the graphs that contain a rayless tree containing a given set $U$ of vertices: they are precisely the graphs $G$ that have a subgraph $H$ in which $U$ is normally spanned and all whose rays are dominated in $H$.
In particular, we obtain the following corollary:

\begin{customcor}{\ref{cor: characterisation graphs with rayless STs}}
Graphs with a normal spanning tree have a rayless spanning tree if and only if all their rays are dominated.
\end{customcor}
\noindent The graphs with a normal spanning tree are well studied and are quite well known: see~\cite{DiestelLeaderNST,jung69,PitzNewNSTobstructions}.

While it is not always possible to find normal trees or rayless trees that are complementary to undominated combs, it turns out that suitable tree-decompositions still serve as complementary structures:
\begin{customthm}{\ref{thm: undominated comb johannes ii}}
\TFAD
\begin{enumerate}
    \item $G$ contains an \dblue{undominated comb} \at $U$;
    \item $G$ has a \dblue{star-decomposition} with finite adhesion sets such that $U$ is contained in the central part and all undominated ends of $G$ live in the leaves' parts.
\end{enumerate}
Moreover, we may assume that the adhesion sets of the tree-decomposition in~\emph{(ii)} are connected.
\end{customthm}

As discussed above, rayless trees are in general too strong to serve as complementary structures for undominated combs.
It turns out that less specific structures than rayless trees, subgraphs all of whose rays are dominated, yield another complementary structure for undominated combs:

\begin{customthm}{\ref{thm: undominated comb johannes}}
\TFAD
\begin{enumerate}
    \item $G$ contains an \dblue{undominated comb} \at $U$;
    \item $G$ has a connected \dblue{subgraph} that contains $U$ and all whose rays are dominated in~it.
\end{enumerate}
\end{customthm}

This paper is organised as follows. In Section~\ref{section: rayless tree}, we prove our duality theorem for undominated combs in terms of rayless trees, Theorem~\ref{thm: undominated comb duality}. In Section~\ref{section: applications rayless tree}, we discuss applications of this duality theorem. In Section~\ref{section: full duality theorems}, we provide our two full duality theorems for undominated combs: Theorem~\ref{thm: undominated comb johannes ii} and Theorem~\ref{thm: undominated comb johannes}.

Throughout this paper, $G=(V,E)$ is an arbitrary graph.
We use the graph theoretic notation of Diestel's book~\cite{DiestelBook5}, and we assume familiarity with the tools and terminology described in the first paper of this series~\cite[Section~2]{StarComb1StarsAndCombs}.


\section{Undominated combs and rayless trees}\label{section: rayless tree} 

\noindent In this section, we will consider rayless trees as structures that are complementary to undominated combs. 
As usual, let $G$ be any connected graph and let $U\subseteq V(G)$ be any vertex set.
There are three reasons why rayless trees containing $U$ are good candidates.
First, an undominated comb \at $U$ is more specific than a comb \at $U$ and in \cite[Theorem~1]{StarComb1StarsAndCombs} we proved that rayless normal trees $T\subseteq G$ that contain $U$ are complementary to combs.
Therefore, structures that are complementary to undominated combs should be less specific than such normal trees.

Second, by the star-comb lemma, $G$ containing no undominated comb \at $U$ can be rephrased as follows: for every infinite subset $U'\subset U$ the graph $G$ contains a star \at ~$U'$. 
So combining such stars in a clever way might lead to a rayless tree containing~$U$.

Finally, a graph cannot contain both an undominated comb \at $U$ and a rayless tree containing $U$ at the same time:\begin{lemma}[{\cite[Lemma~2.4]{StarComb1StarsAndCombs}}]\label{lemma:rayless_tree_contains_star_at_U}
If $U$ is an infinite set of vertices in a rayless rooted tree $T$, then $T$ contains a star \at $U$ which is contained in the up-closure of its central vertex in the tree-order of $T$.
\end{lemma}

For $U=V(G)$, Širáň~\cite{siran} conjectured that $G$ having a rayless spanning tree is complementary to $G$ containing an undominated comb \at $U$.
Surprisingly, his conjecture has turned out to be false, as shown by Seymour and Thomas~\cite{endfaithfulCounterexample}.
The counterexample they have found is also a big surprise.
Recall that $T_\kappa$ for a cardinal $\kappa$ denotes the tree all whose vertices have degree $\kappa$.
\begin{theorem}[{\cite[Theorem~1.6]{endfaithfulCounterexample}}]\label{raylessSpanningTreeCounterexample}
There is an infinitely connected,
in particular one-ended, graph $G$ of order $2^{\aleph_0}$ which does not contain a subdivided $K^{\aleph_1}$, such that every spanning tree of $G$ contains a subdivision of $T_{\aleph_1}$.
\end{theorem}

\noindent Indeed, the end of a graph $G$ as in Theorem~\ref{raylessSpanningTreeCounterexample} is dominated as $G$ is infinitely connected, but for $U=V(G)$ the graph does not contain a rayless tree containing~$U$.

A similar counterexample has been obtained independently by Thomassen~\cite{ThomassenEndfaithfulCounterexample}.
Set-theoretic points of view are presented in both \cite{endfaithfulCounterexample} and Kom\-játh's~\cite{KomjathEndfaithfulMartin}. 
Komjáth even gives a positive consistency result under Martin's axiom for graphs $G$ with $<2^{\aleph_0}$ many vertices: \emph{If $\kappa<2^{\aleph_0}$ is a cardinal, MA$(\kappa)$ holds, and $G$ is infinitely connected with $\vert V(G)\vert\le\kappa$, then $G$ has a rayless spanning tree.}

Nevertheless, it is known that requiring $G$ to be countable does suffice to ensure the existence of a rayless spanning tree when $G$ is connected and every end is dominated, giving the following duality:
\begin{theorem}\label{thm:Siran}
Let $G$ be any connected countable graph. Then \tfad
\begin{enumerate}
    \item $G$ contains an undominated comb \at $V(G)$;
    \item $G$ has a rayless spanning tree.
\end{enumerate}
\end{theorem}
\noindent Proofs are due to Polat~\cite{polatFrenchIII,Polat90} and Širáň~\cite{siran}.
Our main result in this section extends Theorem~\ref{thm:Siran}:

\begin{mainresult}\label{thm: undominated comb duality}
Let $G$ be any graph and let $U \subseteq V(G)$ be normally spanned in $G$. Then \tfad 
\begin{enumerate}
    \item $G$ contains an undominated comb \at $U$;
    \item there is a rayless tree $T\subset G$ that contains $U$.
\end{enumerate}
\end{mainresult}
\noindent Note that this extends Theorem~\ref{thm:Siran} twofold: 
On the one hand, we localise the statement to an arbitrary vertex set $U\subseteq V(G)$. 
On the other hand, we extend the statement to the class of all graphs in which $U$ is normally spanned.

While our focus in this paper is to find duality theorems for undominated combs, Polat and Širáň were rather interested in a characterisation of those graphs that have rayless spanning trees. 
The strongest sufficient condition for the existence of a rayless spanning tree, other than Theorem~\ref{thm: undominated comb duality} (to the knowledge of the authors), is due to Polat~\cite{Polat1997}: \emph{If every end of a connected graph $G$ is dominated and $G$ contains no subdivided $T_{\aleph_1}$, then $G$ has a rayless spanning tree.}
His result does not imply our Theorem~\ref{thm: undominated comb duality}, for example consider $G$ to be the graph obtained from $T_{\aleph_1}$ by completely joining an arbitrarily chosen root to all other nodes, and $U=V(G)$. However, as a corollary of Theorem~\ref{thm: undominated comb duality}, we obtain a full characterisation of the graphs that have rayless spanning trees. Our characterisation even takes an arbitrary vertex set $U\subseteq V(G)$ into account:
\begin{corollary}
Let $G$ be any graph. Then the following assertions are equivalent:
\begin{enumerate}
    \item There is a rayless tree $T\subset G$ that contains $U$;
    \item $G$ has a subgraph $H$ in which  $U\subseteq V(H)$ is normally spanned and all whose rays are dominated in $H$.\qed
\end{enumerate}
\end{corollary}
\noindent If the graph $G$ itself has a normal spanning tree, then our characterisation simplifies as follows:

\begin{maincorollary}\label{cor: characterisation graphs with rayless STs}
Graphs with a normal spanning tree have a rayless spanning tree if and only if all their rays are dominated.\qed
\end{maincorollary}

This section is organised as follows. In Section~\ref{section: a first approximation} we will prove Theorem~\ref{thm: undominated comb duality} for normally spanned graphs. 
Then, in Section~\ref{section: proof of theorem 1}, we will deduce Theorem~\ref{thm: undominated comb duality}. 
\subsection{Proof for normally spanned graphs}\label{section: a first approximation}
As a first approximation to Theorem~\ref{thm: undominated comb duality} we prove the following: 
\begin{theorem}\label{thm: undominated comb duality approximation}
Let $G$ be any normally spanned graph and let $U\subseteq V(G)$ be any vertex set. Then \tfad
\begin{enumerate}
    \item $G$ contains an undominated comb \at $U$;
    \item $G$ contains a rayless tree that contains $U$.
\end{enumerate}
\end{theorem}
Our proof consists of three key ideas, organised in three lemmas: Lemma~\ref{lemma: closure U equals closure T}, Lemma~\ref{lemma: Abs U plus dominating vts equals Abs U} and  Lemma~\ref{lemma: breadth first search tree}. 

\begin{lemma}[{\cite[Lemma~2.13]{StarComb1StarsAndCombs}}]\label{lemma: closure U equals closure T}
Let $G$ be any graph. If $T\subset G$ is a rooted tree that contains a vertex set~$W$ cofinally, then $\Abs{T}=\Abs{W}$.
\end{lemma}

\begin{lemma}\label{lemma: Abs U plus dominating vts equals Abs U}
Let $G$ be any graph and let $U\subseteq V(G)$ be any vertex set. If $\hat{U}$ is the superset of $U$ also containing all the vertices dominating an end in the closure of $U$, then~$\Abs{\hat{U}}=\Abs{U}$.
In particular, $\Abs{U'}=\Abs{U}$ for all vertex sets $U'$ with $U\subset U'\subset \hat{U}$ and $\hat{U}$ contains all the vertices dominating an end in the closure of~$\hat{U}$.
\end{lemma}

\begin{proof}
Every end in the closure of $U$ is contained in the closure of $\hat{U}$ because $\hat{U}$ contains $U$. 
For the other inclusion consider any end $\omega$ in the closure of $\hat{U}$. 
Given a finite vertex set $X\in \cX$ we show that $C(X,\omega)$ contains a vertex from $U$. 
Fix a comb \at $\hat{U}$ and with spine in $\omega$, and pick any tooth $v$ of the comb in the component $C(X,\omega)$ of $G-X$. Then either $v$ is contained in $U$, or $v$ dominates an end $\omega'$ in the closure of $U$ in which case $U$ must meet $C(X,\omega')=C(X,\omega)$.
Therefore, $C(X,\omega)$ meets $U$ for all $X\in\cX$, and so $\omega$ lies in the closure of $U$. 
\end{proof}

\noindent For our last key lemma, we shall need the following result of Jung (cf.~\cite[Theorem~3.5]{StarComb1StarsAndCombs}):

\begin{theorem}[Jung]\label{thm: NT and dispersed sets} 
Let $G$ be any graph. A vertex set $W\subset V(G)$ is normally spanned in $G$ if and only if it is a countable union of dispersed sets. In particular, 
$G$ is normally spanned if and only if $V(G)$ is a countable union of dispersed sets.
\end{theorem}

\begin{lemma}\label{lemma: breadth first search tree}
Let $G$ be any graph and let $U\subseteq V(G)$ be normally spanned. If every end in the closure of $U$ is dominated by some vertex in~$U$, then there is a rayless tree $T\subset G$ containing $U$.  
\end{lemma}
Normal trees follow the concept of depth-first search trees. Speaking informally, all ends of $G$ are `far away' from the perspective of any fixed vertex. This is why normal spanning trees grow towards the ends of the underlying graph in the sense that they contain (precisely) one normal ray from every end. We, however, seek to avoid having any rays in our tree. This is why our construction of a rayless tree containing $U$ will follow the opposite concept to depth-first search trees, namely that of breadth-first search trees. 

\begin{proof}[{Proof of Lemma~\ref{lemma: breadth first search tree}}] 
First we choose a well-ordering of $U$ all whose proper initial segments are dispersed: By Theorem~\ref{thm: NT and dispersed sets}, we have that $U$ is a countable union $\bigcup_{n\in\N}U_n$ of, say pairwise disjoint, dispersed sets $U_n$. Choose a well-ordering ${\preceq_n}$ of every vertex set $U_n$. 
Given $u,u'\in U$ with $u\in U_m$ and $u'\in U_n$, we write $u\preceq u'$ if either $m<n$ or $m=n$ with $u\preceq_m u'$ holds. 
It is straightforward to show that $\preceq$ defines a well-ordering of $U$ that is as desired.
From now on we view $U$ as well-ordered set~$(U,\preceq)$.

We recursively construct an ascending sequence $(T_\alpha)_{\alpha<\kappa}$ of rooted trees $T_\alpha$ sharing their root and satisfying that the overall union of the $T_\alpha$ is a rayless tree containing $U$. 
Let $T_0$ be the tree consisting of and rooted in the smallest vertex of $U$. 
In a limit step $\beta>0$ we let $T_\beta$ be the tree $\bigcup\, \{\,T_\alpha\mid \alpha <\beta\,\}$. 
In a successor step $\beta =\alpha +1$ we terminate and set $\kappa=\beta$ if $U$ is included in $T_\alpha$. 
Otherwise we let $u$ be the smallest vertex in $U\setminus V(T_\alpha)$. 
Following the concept of a breadth-first search tree, among all $u$--$T_{\alpha}$ paths fix one $P_{\beta}$ whose endvertex in $T_\alpha$ has minimal height in $T_{\alpha}$.
We obtain $T_{\beta}$ from $T_{\alpha}$ by adding the path $P_{\beta}$.

Let $T$ be the overall union of the trees $T_\alpha$, i.e., $T:=\bigcup\,\{\,T_\alpha \mid \alpha <\kappa\,\}$. 
Then $T$ is a rooted tree that contains $U$ cofinally. 
It remains to check that $T$ is rayless. Suppose for a contradiction that $R$ is a ray in $T$ starting in the root, say. 
By Lemma~\ref{lemma: closure U equals closure T} the end of the ray $R$ is contained in the closure of $U$. 
As all ends in $\Abs U$ are dominated by vertices in $U$, we find a vertex $u^* \in U$ dominating~$R$. 
Let $P_{\alpha^*}$ be the path from the construction of $T$ that added~$u^*$.

We claim that every tree $T_{\alpha}$ meets $R$ in a finite initial subpath. 
This can be seen as follows.
Since all proper initial segments of $U$ are dispersed, by Lemma~\ref{lemma: closure U equals closure T} it suffices to show that every $T_\alpha$ with $\alpha>0$ contains a subset of such a segment cofinally.
A transfinite induction on $\alpha$ shows that for $T_\alpha$ this subset may be chosen as the set of starting vertices of the paths $P_\xi$ with $\xi\le\alpha$ a successor ordinal while the proper initial segment may be chosen as the  down-closure in $U$ of the starting vertex of $P_{\alpha+1}$.
Here we remark that $\alpha+1<\kappa$ for all $\alpha<\kappa$ (i.e.\ $\kappa$ is a limit ordinal): indeed, by our assumption that $R\subset T$ we know that the vertex set $U$ is not dispersed and, therefore, meets infinitely many $U_n$.

Finally, we derive the desired contradiction.
Fix $\beta>\alpha^*$ so that the endvertex $x$  of $P_{\beta+1}$ in $T_\beta$ has larger height than  $u^*$ has in $T_\beta$ and so that $P_{\beta+1}$ contains an edge of $R$. Let $u$ be the first vertex of $P_{\beta+1}$, i.e., the smallest vertex in $U\setminus V(T_\beta)$. Note that the first vertex $w$ of $P_{\beta+1}$ that is contained in $R$ is distinct from $x$. (Also see Figure~\ref{fig: contradiction rayless tree}.)
\begin{figure}[ht]
    \centering
    \def\svgwidth{.5\columnwidth}
    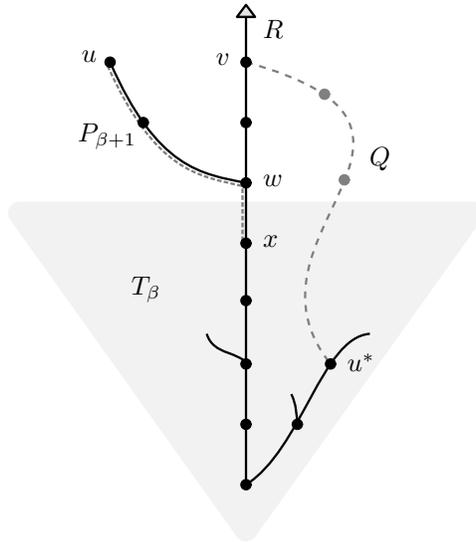
    \caption{The situation in the last paragraph of the proof of Lemma~\ref{lemma: breadth first search tree}.}
    \label{fig: contradiction rayless tree}
\end{figure}
As $u^*$ dominates $R$ we find an infinite set $\mathcal{Q}$ of $u^*$--$R$ paths in $G$ such that distinct paths in $\mathcal{Q}$ only meet in $u^*$. All but finitely many paths in $\mathcal{Q}$ meet $T_{\beta+1}$ precisely in $u^*$: Otherwise the end of $R$ is contained in the closure of $T_{\beta+1}$ contradicting that the vertex set of $T_{\beta+1}$ is dispersed. 
Fix a path $Q\in \mathcal{Q}$ meeting $T_{\beta+1}$ precisely in $u^*$ and having its endvertex $v$ in~$\mathring{w}R$.
We conclude that $uP_{\beta+1} w R v Q u^*$ would have been a better choice than $P_{\beta+1}$ in the construction of $T_{\beta+1}$ (contradiction).
\end{proof}

\begin{proof}[Proof of Theorem~\ref{thm: undominated comb duality approximation}]
By Lemma~\ref{lemma:rayless_tree_contains_star_at_U} at most one of (i) and (ii) holds at a time.
To verify that least one of (i) and (ii) holds, we show $\neg$(i)$\to$(ii). 
By Lemma~\ref{lemma: Abs U plus dominating vts equals Abs U} we may assume that $U$ contains all vertices dominating an end in the closure of $U$, and by Lemma~\ref{lemma: breadth first search tree} there is a rayless tree $T\subset G$ that contains $U$.
\end{proof}

\subsection{Deducing our duality theorem in terms of rayless trees}\label{section: proof of theorem 1}
Let us analyse why the proof of our duality theorem for undominated combs in terms of rayless trees for normally spanned graphs, Theorem~\ref{thm: undominated comb duality approximation}, does not immediately give a proof for arbitrary graphs. 
For this, consider any graph $G$ and let $U\subseteq V(G)$ be any vertex set. 
Furthermore, suppose that there is a normal tree $T\subset G$ that contains $U$ and that $G$ contains no undominated comb \at $U$. 
In the proof of Theorem~\ref{thm: undominated comb duality approximation} we assume without loss of generality that $U$ contains all the vertices dominating an end in the closure of $U$. 
This is possible because, by Lemma~\ref{lemma: Abs U plus dominating vts equals Abs U}, adding all the vertices to $U$ that dominate an end in the closure of $U$ does not change the set $\Abs{U}$ of ends in the closure of $U$. 
However, after adding all these vertices it may happen---in contrast to the situation in the proof of Theorem~\ref{thm: undominated comb duality approximation} where $G$ has a normal spanning tree---that $U$ is no longer normally spanned in $G$ (e.g.\ consider any countably infinite set $U$ of vertices in an uncountable complete graph). 
And $U$ being normally spanned in $G$ is a crucial requirement of the lemma that yields the desired rayless tree, Lemma~\ref{lemma: breadth first search tree}.

But maybe adding all the vertices that dominate an end in the closure of $U$ and maintaining that $U$ is normally spanned was too much to ask.
Indeed, Lemma~\ref{lemma: breadth first search tree} only requires that $U$ contains for every end $\omega\in\Abs{U}$ at least one vertex dominating $\omega$, and adding just one dominating vertex for every end $\omega$ might preserve the property of $U$ being normally spanned in $G$. The following example shows that this is in general false:
\begin{example}
Let $G$ be a \emph{binary tree with tops}, i.e., let $G$ be obtained from the rooted infinite binary tree $T_2$ by adding for every normal ray $R$ of $T_2$ a new vertex $v_R$, its \emph{top}, that is joined completely to $R$ (cf.~Diestel and Leader's \cite{DiestelLeaderNST}). Let $U$ be the vertex set of $T_2$. 
Then $\Abs{U} = \Omega(G)$ and every end  $\omega$ is dominated precisely by the top that was added for the unique normal ray of $T_2$ that is contained in $\omega$. Hence adding for every end in $\Abs{U}$ a vertex dominating it to $U$ results in the whole vertex set of $G$. However, as pointed out in \cite{DiestelLeaderNST}, the graph $G$ does not have a normal spanning tree.\end{example}

Our way out is to work in a suitable contraction minor, which requires some preparation: 
Let $H$ and $G$ be any two graphs. We say that $H$ is a contraction minor of $G$ with \emph{fixed branch sets} if an indexed collection of branch sets $\{\,V_x\mid x\in  V(H)\,\}$ is fixed to witness that $G$ is an $IH$. In this case, we write $[v]=[v]_H$ for the branch set $V_x$ containing a vertex $v$ of $G$ and also refer to $x$ by $[v]$.
Similarly, we write $[U]=[U]_{H}:=\{\,[u]\mid u\in U\,\}$ for vertex sets $U\subseteq V(G)$.

\begin{lemma}\label{lemma: rayless tree containing U if contraction minor has}
Let $G$ be any graph and let $H$ be any contraction minor of $G$ with fixed branch sets that induce subgraphs of $G$ with rayless spanning trees. 
Furthermore, let $U\subset V(G)$ be any vertex set.
If $H$ contains a rayless tree that contains $[U]$, then $G$ contains a rayless tree that contains $U$.
\end{lemma}

\begin{proof}
Let $T\subset H$ be a rayless tree that contains $[U]$. 
Fix for every branch set $W\in [V(T)]$ a rayless spanning tree $T_{W}$ in the subgraph that $G$ induces on $W$.
Furthermore, select one edge $e_f\in E_G(t_1,t_2)$ for every edge $f=t_1t_2\in T$.
It is straightforward to show that the union of all the trees $T_{W}$ plus all the edges $e_f$ is a rayless tree in $G$ that contains $U$.
\end{proof}

Let $H$ be a contraction minor of a graph $G$ with fixed branch sets. A subgraph $G'=(V',E')$ of $G$ can be passed on to $H$ as follows. Take as vertex set the set $[V']$ and declare $W_1W_2$ to be an edge whenever $E'$ contains an edge between $W_{1}$ and $W_{2}$. We write $[G']=[G']_H$ for the resulting subgraph of $H$ and call it the graph that is obtained by \emph{passing on} $G'$ to $H$.
If every vertex $W\in [V']$ meets $V'$ in precisely one vertex, then we say that $G'$ is \emph{properly passed on} to $H$. Note that if $G'$ is properly passed on to $H$, then $[G']$ and $G'$ are isomorphic. 

\begin{lemma}\label{lemma: properly passed on nt}
Let $H$ be a contraction minor of a graph $G$ with fixed branch sets and let $T\subseteq G$ be a tree that is normal in $G$. If $T$ is properly passed on to $H$, then $[T]\subset H$ is a tree that is normal in $H$. 
\end{lemma}
\begin{proof}
Since $T$ is properly passed on to $G$ we have that $T$ and $[T]$ are isomorphic as witnessed by the bijection $\varphi$ that maps every vertex $t\in T$ to $[t]$. 
In order to see that $[T]$ is normal in $H$ when it is rooted in $[r]$ for the root $r$ of $T$, consider any $[T]$-path $W_0\dots W_k$ in $[H]$. 
Using that branch sets are connected, it is straightforward to show that there is $T$-path in $G$ between the two vertices $\varphi^{-1}(W_0)$ and $\varphi^{-1}(W_k)$ of $T$. Hence $W_0$ and $W_k$ must be comparable in $[T]$.
\end{proof}

We need two more lemmas for the proof of Theorem~\ref{thm: undominated comb duality}.
Recall that the \emph{generalised up-closure} $\guc{x}$ of a vertex $x\in T$ is the union of $\uc{x}$ with the vertex set of $\bigcup \cC(x)$, where the set $\cC(x)$ consists of those components of $G-T$ whose neighbourhoods meet $\uc{x}$.
\begin{lemma}[{\cite[Lemma~2.10]{StarComb1StarsAndCombs}}]\label{lemma: separation properties normal tree}
Let $G$ be any graph and $T\subset G$ any normal tree.
\begin{enumerate}
    \item Any two vertices $x,y\in T$ are separated in $G$ by the vertex set $\dc{x}\cap\dc{y}$.
    \item Let $W\subset V(T)$ be down-closed. Then the components of $G-W$ come in two types: the components that avoid $T$; and the components that meet $T$, which are spanned by the 
    sets $\guc{x}$ with $x$ minimal in $T-W$. 
\end{enumerate}
\end{lemma}

\begin{lemma}[{\cite[Lemma~2.11]{StarComb1StarsAndCombs}}]\label{NormalTreeNormalRay}
If $G$ is any graph and $T\subset G$ is any normal tree,
then every end of $G$ in the closure of $T$ contains exactly one normal ray of $T$.
Moreover, sending these ends to the normal rays they contain defines a bijection between $\Abs{T}$ and the normal rays of~$T$.
\end{lemma}

\begin{proof}[Proof of Theorem~\ref{thm: undominated comb duality}]
Given a normally spanned vertex set $U\subset V(G)$ we have to show that \tfad
\begin{enumerate}
    \item $G$ contains an undominated comb \at $U$;
    \item $G$ contains a rayless tree that contains $U$.
\end{enumerate}
\noindent By Lemma~\ref{lemma:rayless_tree_contains_star_at_U} at most one of (i) and (ii) holds at a time.
To verify that at least one of (i) and (ii) holds, we show $\neg$(i)$\to$(ii). For this, we may assume by Lemma~\ref{lemma: closure U equals closure T} that $U$ is the vertex set of a normal tree $T\subset G$.
In the following we will find a contraction minor $H$ of $G$ with fixed branch sets $V_x$ such that:
\begin{itemize}
    \item[\textbf{--}] all $G[V_x]$ have rayless spanning trees; 
    \item[\textbf{--}] $T$ is properly passed on to $H$;
    \item[\textbf{--}] and every end of $H$ in the closure of $[T]\subset H$ is dominated in $H$ by some vertex of $[T]$.
\end{itemize}
Before we prove that such $H$ exists, let us see how to complete the proof once $H$ is found. 
By Lemma~\ref{lemma: properly passed on nt}, the tree $[T]$ is normal in $H$, and it has vertex set $[U]$ because $V(T)=U$. 
So, by Lemma~\ref{lemma: breadth first search tree}, the graph $H$ contains a rayless tree that contains $[U]$.
Finally, by Lemma~\ref{lemma: rayless tree containing U if contraction minor has}, this rayless tree in $H$ containing $[U]$ gives rise to a rayless tree in $G$ containing $U$ as desired. 

In order to construct $H$, fix for every normal ray $R$ of $T$ a vertex $v_R$ dominating $R$ in $G$. 
Let $\mathcal{R}$ be the set of all normal rays $R$ of $T$ for which $v_R$ is contained in a component $C_R$ of $G-T$.  
Note that the down-closure of the neighbourhood of each $C_R$ is $V(R)$ due to the separation properties of normal trees (Lemma~\ref{lemma: separation properties normal tree}). 
Thus, we have $C_R\neq C_{R'}$ for distinct normal rays $R,R'\in \mathcal{R}$. 
Fix a $v_R$--$R$ path $P_R$ for every $R\in\mathcal{R}$. 
Then the overall union of the paths $P_R$ is a forest of subdivided stars, each having its centre on $T$. 
Let us refer by $S_R$ to the subdivided star that contains $v_R$ for $R\in \mathcal{R}$, i.e., $S_R$ is the union of all the paths $P_{R'}$ that contain the last vertex of $P_R$ and this last vertex is the centre of $S_R$. 
Let $H$ be the contraction minor of $G$ with fixed branch sets defined as follows: 
if $v$ is contained on a path $P_R$, then put $[v]:=S_R$; otherwise let $[v]:=\{v\}$.
Then, in particular, every branch set of $H$ induces a subgraph of $G$ that has a rayless spanning tree.

As every star $S_R$ meets $T$ precisely in its centre, the tree $T$ is properly passed on to $H$.
By Lemma~\ref{lemma: properly passed on nt}, the tree $[T]\subseteq H$ is normal in $H$ and $V([T])=[U]$ since $V(T)=U$.
And by Lemma~\ref{NormalTreeNormalRay} it remains to show that every normal ray of $[T]$ is dominated in $H$ by some vertex of~$[T]$.
For this, we consider three cases. In all three cases, fix any normal ray $R\subset T$ and some collection $\mathcal{P}$ of infinitely many $v_R$--$R$ paths in $G$ meeting precisely in $v_R$.

First assume that $R\in \mathcal{R}$. Note that only finitely many of the paths in $\mathcal{P}$ meet $\mathring{v_R}P_R$, without loss of generality none. 
Then all graphs $[P]\subset H$ with $P\in \mathcal{P}$ are $[v_R]$--$[R]$ paths that meet only in $[v_R]$. 
This shows that $[v_R]\in [T]$ dominates~$[R]$ in~$H$.  

Second, suppose that $R\notin \mathcal{R}$ and that every branch set of $H$ other than $ [v_R]$ meets only finitely many of the paths in $\mathcal{P}$. 
By thinning out $\mathcal{P}$ we may assume that every branch set other than $[v_R]$ meets at most one of the paths in $\mathcal{P}$. Then the connected graphs $[P]$ with $P\in \mathcal{P}$ pairwise meet in $[v_R]$ but nowhere else and all contain a vertex of $[R]$ other than $[v_R]$. 
Taking one $[v_R]$--$([R]-[v_R])$ path inside each $[P]$ yields a fan witnessing that $[v_R]\in [T]$ dominates~$[R]$ in $H$.

Finally, suppose that $R\notin\mathcal{R}$ and that some branch set $S\neq [v_R]$ of $H$ meets infinitely many of the paths in $\mathcal{P}$, say all of them. 
We write $c$ for the centre of $S$. Without loss of generality none of the paths in $\mathcal{P}$ contains $c$. Also note that $c$ is contained in $V(R)$ as otherwise all the paths in $\mathcal{P}$ need to pass through the finite down-closure of $c$ in $T$ in vertices other than $v_R$.
Let $\mathcal{R}'$ be the collection of normal rays  of $T$ that satisfies $S=\bigcup\,\{\,V(P_{R'})\mid R'\in \mathcal{R}'\,\}$. 
For every $v_R$--$R$ path $P\in \mathcal{P}$ let $v_P$ be the last vertex on $P$ that is contained in $S$, let $w_P$ be the first vertex on $P$ after $v_P$ in which $P$ meets $T$ and let $Q_P$ be the unique $w_P$--$R$ path in~$T$. (See Figure~\ref{fig: proof duality thm rayless tree}.)
For every path $P\in \mathcal{P}$ let $P'=P'(P):= v_P Pw_PQ_P$, and let $\mathcal{P}'=\mathcal{P}'(\mathcal{P}):= \{\,P'\mid P\in \mathcal P\,\}$. 

\vspace*{0.005cm}
\begin{figure}[h]
    \centering
    \def\svgwidth{.6\columnwidth}
    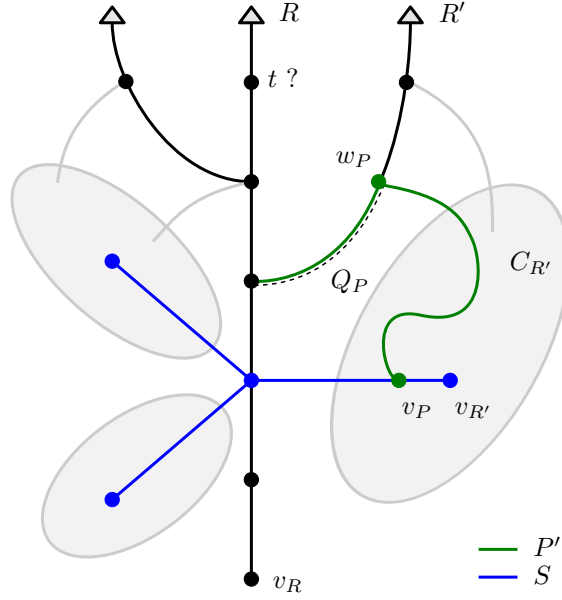
    \caption{The final case in the proof of our duality theorem for  undominated combs in term of rayless trees.}
    \label{fig: proof duality thm rayless tree}
\end{figure}

Each path $P_{R'}\mathring{c}\subset S$ with $R'\in\mathcal{R}'$ meets only finitely many paths from $\mathcal{P}'$, and these latter paths are precisely the paths in $\mathcal{P}'$ that meet $C_{R'}$: 
This is because every path in $\mathcal{P}'$ that meets $C_{R'}$ starts in a vertex $v_P\in C_{R'}$ and after leaving $C_{R'}$ only traverses through vertices of $T$.
Therefore, by replacing $\mathcal{P}$ with an infinite subset of $\mathcal{P}$, we can see to it that every component $C_{R'}$ with $R'\in\mathcal{R}'$ meets at most one of the paths in the then smaller set $\mathcal{P}'=\mathcal{P}'(\mathcal{P})$.
In countably many steps we fix paths $P_1',P_2',\dots$ in $\mathcal{P}'$ so that their last vertices are pairwise distinct: In order to see that this is possible suppose for a contradiction that $t\in R$ is maximal in the tree order of $T$ so that $t$ is the last vertex of a path in $\mathcal{P}'$. 
Note that $R$ together with the paths $v_P P$ with $P\in \mathcal{P}$ forms a comb in $G$. 
Hence infinitely many of the paths $v_P P$ are contained in the same component of $G-\dc{t}$ as some tail of~$R$.
By Lemma~\ref{lemma: separation properties normal tree}, this component is of the form $\guc{t'}$ for the successor $t'$ of $t$ on~$R$.
In particular, we find some $P\in\mathcal{P}$ so that $w_P$ lies above $t'$ in the tree order of~$T$. But then the endvertex of $Q_P$ in $R$ lies above $t'$ and, in particular, above $t$, contradicting the choice of $t$.

So let $P_1',P_2',\dots$ be paths in $\mathcal{P}'$ with pairwise distinct last vertices. 
We show that the paths $P_i'$ give rise to $S$--$[R]$ paths $[P_i']$ in $H$ that form an infinite $S$--$[R]$ fan witnessing that $S$ dominates $[R]$ in $H$.
Every path $P_i'$ is an $S$--$R$ path because every path in $\mathcal{P}'$ is an $S$--$R$ path by the choice of the vertices $v_P$.
Moreover, the paths $P_i'$ are pairwise disjoint: Every path $P_i'$ starts in a component $C_{R'}$.
Using the choice of the vertices $v_P$ with $P\in \mathcal{P}$ as the last vertex on $P$ that is contained in $S$ we have that the $[P_i']$ are $S$--$[R]$ paths of $H$ that only share their first vertex~$S$. Hence the $[P_i']$ form an infinite $S$--$R$ fan in $H$ and we conclude that $S\in [T]$ dominates $[R]$ in $H$.
\end{proof}

\section{Spanning trees reflecting the undominated ends}\label{section: applications rayless tree} 

\noindent In \cite{halin64}, Halin conjectured that every connected graph has a spanning tree that is end-faithful for all its ends.
However, Seymour and Thomas' counterexample in Theorem~\ref{raylessSpanningTreeCounterexample} shows that his conjecture is in general false. 
Recently, Carmesin~\cite{carmesin2014all} amended Halin's conjecture by proving the follwing: 
\begin{theorem}[Carmesin~2014]\label{thm: every graph has st displaying undominated ends}
Every connected graph $G$ has a spanning tree that is end-faithful for the undominated ends of $G$.
\end{theorem}
\noindent Carmesin pointed out that his theorem is best possible in that it becomes false when one replaces `is end-faithful for' with the more specific `reflects' in its wording:
by Theorem~\ref{raylessSpanningTreeCounterexample} there are connected graphs without rayless spanning trees all whose rays are dominated. 
Characterising the graphs that have spanning trees reflecting their undominated ends has remained an open problem.

In order to prove Theorem~\ref{thm: every graph has st displaying undominated ends}, Carmesin developed the following theorem: 

\begin{theorem}[Carmesin, {\cite[Theorem~2.17]{StarComb1StarsAndCombs}}]\label{thm: tdc displaying undom ends}
Every connected graph $G$ has a rooted tree-decomposition with upwards disjoint finite connected separators that displays the undominated ends~of~$G$.
\end{theorem} 

\noindent Here, we recall from~\cite{StarComb1StarsAndCombs} that a tree-decomposition $(T,\cV)$ of a given graph $G$ with finite separators \emph{displays} a set $\Psi$ of ends of $G$ if $\tau$ restricts to a bijection $\tau\rest\Psi\colon\Psi\to\Omega(T)$ between $\Psi$ and the end space of $T$  and maps every end that is not contained in $\Psi$ to some node of $T$, where $\tau\colon\Omega(G)\to\Omega(T)\sqcup V(T)$ maps every end of $G$ to the end or node of $T$ which it corresponds to or lives at, respectively.

Call a connected graph $G$ \emph{\kn } if it has a tree-decomposition with pairwise disjoint finite and connected separators that displays the undominated ends of~$G$, i.e., if $G$ has a tree-decomposition as in Theorem~\ref{thm: tdc displaying undom ends} with the strengthening that all separators are pairwise disjoint.
Our aim in this section is twofold.
First, we prove Theorem~\ref{thm: reformulation special spanning tree} below which provides an existence criterion for spanning trees that reflect the undominated ends of a given graph. 
Second, we characterise in Theorem~\ref{thm: characterisation trees with finite fundamental cuts}~(i) the spanning trees of finitely separable graphs that reflect the undominated ends, and we establish in Theorem~\ref{thm: characterisation trees with finite fundamental cuts}~(ii) that every connected finitely separable graph has such a tree.

Our existence criterion for spanning trees that reflect the undominated ends can be formulated locally:

\begin{theorem}\label{thm: reformulation special spanning tree}
Let $G$ be any \kn\ graph and let $U\subseteq V(G)$ be normally spanned. Then there is a tree $T\subset G$ that contains $U$ and reflects the undominated ends in the closure of $U$. 
\end{theorem}

\noindent Our proof of Theorem~\ref{thm: reformulation special spanning tree} requires some preparation. 

Recall that a rooted tree-decomposition $(T,\cV)$ of a graph $G$ \emph{covers} a vertex set $U\subset V(G)$ \emph{cofinally} if the set of nodes of $T$ whose parts meet $U$ is cofinal in the tree-order of $T$.
\begin{theorem}\label{thm: tdc displaying undom ends in cl of U}
Let $G$ be any \kn\ graph and let $U\subseteq V(G)$ be any vertex set. 
Then $G$ has a rooted tree-decomposition with pairwise disjoint finite connected separators that displays the undominated ends of $G$ that lie in the closure of $U$. Moreover, the tree-decomposition can be chosen so that it covers $U$ cofinally. 
\end{theorem}

\begin{proof}
Since $G$ is \kn , we find a tree-decomposition $(T,\cV)$ of $G$ with pairwise disjoint finite connected separators that displays the undominated ends of $G$. Consider $T$ rooted in an arbitrary node. 
Let $U'$ be the set of vertices of $T$ whose parts meet $U$ and let $T'$ be the subtree of $T$ obtained by taking the down-closure of $U'$ in~$T$. 
Then we let $(T,\alpha)$ be the $\Sinf$-tree corresponding to $(T,\cV)$, so $(T',\alpha\rest\vE(T')\,)$ is an $\Sinf$-tree that induces the desired tree-decomposition.
\end{proof}

Our construction of a tree reflecting the undominated ends in the closure of a given set of vertices will employ a contraction minor $H$ of the underlying graph $G$. 
The following notation will help us to translate between the endspace of $G$ and that of $H$.
Consider a contraction minor $H$ of a graph $G$ with fixed finite branch sets.
Every direction $f$ of $G$ defines a direction $[f]$ of $H$ by letting $[f](X):=[f(\bigcup X)]$  for every finite vertex set $X\subseteq V(H)$.
In fact, it its straightforward to check that every direction of $H$ is defined by a direction of $G$ in this way:

\begin{lemma}
Let $H$ be a contraction minor of a graph $G$ with fixed finite branch sets.
Then the map $f\mapsto [f]$ is a bijection between the directions of $G$ and the directions of $H$.\qed
\end{lemma}
\noindent This one-to-one correspondence then combines with the well-known one-to-one correspondence between the directions and ends of a graph (see \cite[Theorem~2.7]{StarComb1StarsAndCombs}), giving rise to a bijection $\omega\mapsto [\omega]$ between the ends of $G$ and the ends of~$H$.
The natural one-to-one correspondence between the two end spaces extends to other aspects of the graphs and their ends:

\begin{lemma}\label{lemma: end space equivalence contraction minor}
Let $H$ be a contraction minor of a graph $G$ with fixed finite branch sets, let $\omega$ be an end of $G$ and let $U\subset V(G)$ be any vertex set.
Then $\omega$ lies in the closure of $U$ in $G$ if and only if $[\omega]$ lies in the closure of $[U]$ in $H$; and $\omega$ is dominated in $G$ if and only if $[\omega]$ is dominated in $H$.
\end{lemma}
\noindent We remark that this extends \cite[Exercise~82~(i)]{DiestelBook5}.
\begin{proof}
Write $f_\omega$ for the direction of $G$ that corresponds to $\omega$. Then the following statements are equivalent: 
\begin{enumerate}
    \item $\omega$ lies in the closure of $U$ in $G$;
    \item  $f_\omega(X)$ meets $U$ for every finite vertex set $X\subset V(G)$;
    \item  $[f_\omega](X)$ meets $[U]$ for every finite vertex set $X\subseteq V(H)$;
    \item $[\omega]$ lies in the closure of $[U]$ in $H$.
\end{enumerate}
\noindent Indeed, one easily verifies  (i)$\leftrightarrow$(ii)$\leftrightarrow$(iii)$\leftrightarrow$(iv).

This establishes that the end~$\omega$ of $G$ lies in the closure of $U$ in $G$ if and only if $[\omega]$ lies in the closure of $[U]$ in~$H$. Similarly, it is straightforward to check that the following statements are equivalent for any vertex $v$ of $G$ (except for (iii)$\to$(ii) which we will verify in detail):
\begin{enumerate}
    \item there is a vertex $z\in[v] $ that dominates $\omega$ in $G$;
    \item there is a vertex $z\in [v]$ such that $z\in f_\omega(X)$ for every finite vertex set $X\subseteq V(G)\setminus \{z\}$;
    \item $[v]\in [f_\omega](X)$ for every finite vertex set $X\subseteq V(H)\setminus \{[v]\}$;
    \item $[v]$ dominates $[\omega]$ in $H$.
\end{enumerate}
\noindent 
To see (iii)$\to$(ii) we show $\neg$(ii)$\,{\to}\,{\neg}$(iii).
Since (ii) fails, there is for every vertex $z\in [v]$ a finite vertex set $X_z\subset V(G)\setminus\{z\}$ such that $z$ is not contained in $f_\omega(X_z)$.
Consider the finite vertex set $X:=\bigcup_z X_z$.
Then no $z\in [v]$ is contained in the component $f_\omega(X)$ or is one of its neighbours, because $f_\omega(X)\subset f_\omega(X_z)$ and $z\notin X_z\cup f_\omega(X_z)$.
Hence $[v]\notin [f_\omega]([X'])$ for the neighbourhood $X'$ of $f_\omega(X)$ in $G$ and this neighbourhood avoids~$[v]$.

Therefore the end $\omega$ of $G$ is dominated in $G$ if and only if $[\omega]$ is dominated in~$H$.
\end{proof}

\begin{lemma}\label{GraphsWithNSTminorClosed}
Let $H$ be a contraction minor of a graph $G$ with fixed branch sets and let $U\subset V(G)$ be any vertex set. If $U$ is normally spanned in $G$, then $[U]$ is normally spanned in $H$.
\end{lemma}
\noindent We remark that this is essentially \cite[Lemma~7.2~(b)]{halin2000}.

\begin{proof} Without loss of generality both $G$ and $H$ are connected.
By Theorem~\ref{thm: NT and dispersed sets}, we have that $U$ can be written as a countable union $\bigcup_{n\in\N} U_n$ with every $U_n$ dispersed in $G$. 
Then every vertex set $[U_n]$ is dispersed in $H$, because every comb \at $[U_n]$ in $H$ would give rise to a comb \at $U_n$ in $G$, contradicting that $U_n$ is dispersed in $G$.
Hence $[U]=\bigcup_{n\in\N}[U_n]$ is normally spanned in $H$ by Theorem~\ref{thm: NT and dispersed sets}.
\end{proof}

We need one more lemma for the proof of Theorem~\ref{thm: reformulation special spanning tree}:
\begin{lemma}\label{lemma: Abs U equals Abs U plus separators of tdc}
Let $G$ be any connected graph and let $U\subseteq V(G)$ be any vertex set. If $(T,\cV)$ is a rooted tree-decomposition of $G$ with pairwise disjoint finite connected separators that displays the undominated ends in $\Abs{U}$ and covers~$U$ cofinally, then $\Abs{U}=\Abs{\hat{U}}$ for the superset $\hat{U}$ of $U$ that arises from $U$ by adding all the vertices that lie in the separators of $(T,\cV)$. 
\end{lemma}

\begin{proof}
 The inclusion $\Abs{U}\subset \Abs{\hat{U}}$ holds because $U\subseteq \hat{U}$. For the backward inclusion, consider any end $\omega$ in the closure of $\hat{U}$, and assume for a contradiction that $\omega$ does not lie in the closure of $U$.
 Then $\omega$ lives at a node $t\in T$ because $(T,\cV)$ displays the ends in the closure of $U$.
 Pick a comb in $G$ \at ~$\hat{U}$ and with spine in $\omega$.
 As $\omega$ does not lie in the closure of $U$ we may assume that the comb avoids $U$.
Furthermore, we may assume that every tooth of the comb lies in a separator of $(T,\cV)$ associated with an edge of $T$ at and above $t$.
 Since the separators of $(T,\cV)$ are finite and pairwise disjoint, we may even ensure that no separator contains more than one tooth. 
 As $(T,\cV)$ has connected separators and covers $U$ cofinally, we find infinitely many disjoint paths from the comb to $U$, one starting in each tooth.
 Then the comb together with these paths witnesses that $\omega$ lies in the closure of $U$, a contradiction.
\end{proof}

\begin{proof}[Proof of Theorem~\ref{thm: reformulation special spanning tree}]
Let $G$ be any \kn\ graph and let $U\subseteq V(G)$ be normally spanned.
Let $(\dt,\cV)$ be any rooted tree-decomposition of $G$ with pairwise disjoint finite connected separators such that $(\dt,\cV)$ displays the undominated ends in the closure of $U$ and covers $U$ cofinally. 
And by Lemma~\ref{lemma: Abs U equals Abs U plus separators of tdc} we may assume that $U$ contains all the vertices that are contained in the separators of $(\dt,\cV)$.

We construct a tree $T\subset G$ displaying the undominated ends in the closure of $U$ as follows. 
For every separator $X$ of $(\dt,\cV)$ we pick a spanning tree $T_X$ of $G[X]$.
As all $X$ are finite and pairwise disjoint, so are the $T_X$.
Next, we choose for every part $V_t$ of $(\dt,\cV)$ a rayless tree $T_t$ in $G[V_t]$ containing $U_t:=V_t\cap U$ and extending all the trees $T_X$ for which $X$ is a separator corresponding to some edge incident with $t$, as follows. 
Given $V_t$, we first consider the contraction minor $H_t$ of $G[V_t]$ with fixed branch sets that is obtained from $G[V_t]$ by contracting each $G[X]$ with $X$ a separator induced by an edge of $\dt$ at $t$ to a single \emph{dummy} vertex named $X$.
As $U$ is normally spanned in $G$ it follows by Lemma~\ref{GraphsWithNSTminorClosed} that $[U]_{H}$ is normally spanned in the contraction minor $H$ obtained from $G$ by contracting every $G[X]$ for every separator.
It follows that the vertex sets $[U_t]_{H_t}$ are normally spanned in $H_t\subset H$. 
Furthermore, since $(\dt,\cV)$ has disjoint finite connected separators and displays the undominated ends of $G$ in the closure of $U$, every end of $G[V_t]$ in the closure of $U_t$ in the graph $G[V_t]$ is dominated in $G[V_t]$. Thus, by Lemma~\ref{lemma: end space equivalence contraction minor} every end  of $H_t$ in the closure of $[U_t]$ is dominated in $H_t$.
Hence we may apply Theorem~\ref{thm: undominated comb duality} to $H_t$ and $[U_t]$ to obtain a rayless tree $\tilde{T}_t$ in $H_t$ containing $[U_t]$. 
Then by expanding each dummy vertex $X$ of $\tilde{T}_t$ to $T_X$ we obtain a rayless tree $T_t$ in $G[V_t]$ that contains $U_t$ and extends all these~$T_X$.

Let $T$ be spanned by the down-closure of $U$ in the tree $\bigcup_{t\in \dt}T_t$ with regard to an arbitrary root. We claim that $T$ contains $U$ and reflects the undominated ends in the closure of $U$.
Clearly, $T$ is a tree in $G$ that contains $U$ even cofinally. 
By the star-comb lemma, every tree in $G$ containing $U$ contains for each undominated end in the closure of $U$ a ray from that end.
In particular, $T$ contains a ray from every undominated end in the closure of~$U$.

Next, the tree $T$ contains at most one ray starting in the root for every undominated end in the closure of $U$: Indeed, if $T$ contains two (say) vertex-disjoint rays from the same undominated end $\omega$ in the closure of $U$, then these give rise to a subdivided ladder in $T$ via the trees $T_X$ along any ray of $\dt$ to which $\omega$ corresponds, and the ladder comes with infinitely many cycles, contradicting that $T$ is a tree.

That $T$ contains only rays from ends in the closure of $U$ is a consequence of Lemma~\ref{lemma: closure U equals closure T} and the fact that $T$ contains $U$ cofinally by construction.

Finally, the tree $T$ contains no ray from dominated ends in the closure of $U$, for if $T$ contains a ray from such an end, then the vertex set of that ray intersects some part $V_t$ of $(T,\cV)$ infinitely often, and then Lemma~\ref{lemma:rayless_tree_contains_star_at_U} applied in the rayless tree $T_t$ to that intersection yields infinitely many cycles in the tree $T$.
\end{proof}

\begin{theorem}\label{thm: characterisation trees with finite fundamental cuts} Let $G$ be any graph and let $T\subseteq G$ be any spanning tree. \begin{enumerate}
    \item The fundamental cuts of $T$ are all finite if and only if $G$ is finitely separable and $T$ reflects the undominated ends of $G$.
    \item  If $G$ is finitely separable and connected, then it has a spanning tree all whose fundamental cuts are finite. 
\end{enumerate} 
\end{theorem}

\begin{proof}
(i) For the forward implication suppose that the fundamental cuts of $T$ are all finite. 
First let us see that $G$ is finitely separable. 
For this consider any two distinct vertices $v,w\in V(G)$ and let $e$ be an edge on the unique path between $v$ and $w$ in $T$. Then the fundamental cut of $e$ with respect to $T$ is finite and separates $v$ from $w$ in $G$.

Next, let us show that no ray of $T$ is dominated. 
For this, consider any ray $R\subseteq T$ and any vertex $v\in V(G)$. Let $C$ be the component of $T-v$ that contains a tail of $R$ and let $e\in E(T)$ be the unique edge between $C$ and $v$. 
As the fundamental cut of $e$ with respect to $T$ is finite, and as all the paths of any $v$-$(R-v)$ fan need to pass through this fundamental cut, the vertex $v$ cannot dominate $R$.

The tree $T$ contains a ray from every undominated end, because, by the star-comb lemma, every spanning spanning tree of $G$ does so.
It remains to show that every distinct two ends of $T$ are included in distinct ends of $G$. 
For this consider rays $R, R'\subseteq T$ that belong to distinct ends of $T$. Let $e$ be an edge on a tail of $R$ that does not meet $R'$. 
Then the endvertices of the edges in the finite fundamental cut of $e$ form a finite vertex set that separates a tail of $R$ from a tail of $R'$ in $G$.
Hence $R$ and $R'$ belong to distinct ends of $G$.

For the backward implication suppose that $G$ is finitely separable and that $T$ reflects the undominated ends of $G$. Consider any fundamental cut $F_e$ of an edge $e\in E(T)$ with respect to $T$.
Write $T_1$ and $T_2$ for the two components of $T-e$. Then $F_e$ consists of the $T_1$--$T_2$ edges of $G$.
Suppose for a contradiction that $F_e$ is infinite.
Then $F_e$ has infinitely many endvertices in at least one of $T_1$ and $T_2$.
Let us write $X_i$ for the set of endvertices that $F_e$ has in $T_i$ for $i=1,2$.
We consider two cases and derive contradictions for both of them.

In the first case, some vertex $x\in X_i$ is incident with infinitely many edges of~$F_e$, say for $i=1$.
Then, as $G$ is finitely separable, applying the star-comb lemma in $T_2$ to the infinitely many endvertices that these edges have in $T_2$ must yield a comb whose spine is then dominated by $x$ in $G$, contradicting that $T$ reflects the undominated ends of $G$.

In the second case, every vertex of $G$ is incident with at most finitely many edges from $F_e$.
Then $F_e$ contains an infinite matching of an infinite subset of $V(T_1)$ and an infinite subset of $V(T_2)$.
First, we apply the star-comb lemma in $T_1$ to the endvertices of this matching.
This yields either a star or a comb, and we write $U_1$ for its attachment set.
Then we apply the star-comb lemma in $T_2$ to those vertices that are matched to $U_1$.
Since $G$ is finitely separable, we cannot get two stars.
Like in the first case, we cannot get one star and one comb.
So we must get two combs.
But then $T$ contains two rays that are equivalent in $G$, contradicting that $T$ reflects some set of ends of $G$.

(ii)  By (i) it suffice to show that $G$ has a spanning tree that reflects its undominated ends.
Bruhn and Diestel~\cite[Theorem~6.3]{duality} showed that $G$ has a spanning tree $T$ whose closure in~$\tilde{G}$ does not contain a circle (using their terminology).
We claim that $T$ reflects the undominated ends of~$G$.
For this, we show that 
\begin{enumerate}[label=(\arabic*)]
    \item no ray in $T$ is dominated in~$G$, and that
    \item no two disjoint rays in $T$ are equivalent in~$G$.
\end{enumerate}
Indeed, if $T$ contains a ray that is dominated in $G$ by a vertex~$v$, then that ray is a tail of ray $R\subset T$ that starts in~$v$, so $\overline{R}\subset\overline{T}$ is a circle contradicting the choice of~$T$.
And if $T$ contains two disjoint equivalent rays, then there is a double ray $D\subset T$ that contains both rays, and neither of the two rays is dominated by~(1).
Thus, $\overline{D}\subset\overline{T}$ is a circle contradicting the choice of~$T$.
\end{proof}

\newpage
\section{Duality theorems for undominated combs}\label{section: full duality theorems}

\noindent In this section we prove our two duality theorems for undominated combs in full generality. 
The first theorem is phrased in terms of star-decompositions:
\begin{mainresult}\label{thm: undominated comb johannes ii}
\TFAD
\begin{enumerate}
    \item $G$ contains an undominated comb \at $U$;
    \item $G$ has a star-decomposition with finite separators such that $U$ is contained in the central part and all undominated ends of $G$ live in the leaves' parts.
\end{enumerate}
Moreover, we may assume that the separators of the tree-decomposition in~\emph{(ii)} are connected.
\end{mainresult}

\begin{proof}
Clearly, at most one of (i) and (ii) can hold.

To establish that at least one of (i) and (ii) holds, we show $\neg$(i)$\to$(ii). 
By Theorem~\ref{thm: tdc displaying undom ends} we find a rooted tree-decomposition $(T,\cV)$ of $G$ with upwards disjoint finite connected separators that displays the undominated ends of~$G$.
We let $W\subset V(T)$ consist of those nodes $t\in T$ whose parts $V_t$ meet~$U$.
Then we root $T$ arbitrarily and let $T'$ be the subtree $\dc{W}$ of $T$.
Since $U$ does not have any undominated end of $G$ in its closure, it follows that $T'$ must be rayless.
We obtain the star $S$ from $T$ by contracting $T'$ and all of the components of $T-T'$.
Then we let $(T,\alpha)$ be the $\Sinf$-tree corresponding to $(T,\cV)$, so $(S,\,\alpha\rest \vE(S)\,)$ is an $\Sinf$-tree that induces the desired star-decomposition which even satisfies the `moreover' part.
\end{proof}
The central part of the star-decomposition in Theorem~\ref{thm: undominated comb johannes ii}~(ii) induces a subgraph of $G$ that seems to carry the information that there is no undominated comb \at $U$. Our second duality theorem for undominated combs confirms this suspicion:  
\begin{mainresult}\label{thm: undominated comb johannes}
\TFAD
\begin{enumerate}
    \item $G$ contains an undominated comb \at $U$;
    \item $G$ has a connected subgraph that contains $U$ and all whose rays are dominated in~it.
\end{enumerate}
\end{mainresult}
\begin{proof}
To see that at most one of (i) and (ii) holds, consider any connected subgraph $H\subset G$ containing $U$ such that every ray of $H$ is dominated in $H$.
We show that $H$ obstructs the existence of an undominated comb in $G$ \at $U$.
Assume for a contradiction that such a comb exists. Then the undominated end $\omega\in\Omega(G)$ of that comb's spine lies in the closure of $U$, and so applying the star-comb lemma in $H$ to the attachment set $U'\subset U$ of that comb must yield another comb \at $U'$.
But this latter comb is dominated in $H$ by assumption, and at the same time its spine is equivalent in $G$ to the first comb's spine, contradicting that $\omega$ is undominated in $G$. 

To establish that at least one of (i) and (ii) holds, we show $\neg$(i)$\to$(ii). 
Let $(T,\cV)$ be the star-decomposition from Theorem~\ref{thm: undominated comb johannes ii}~(ii) also satisfying the `moreover' part of the theorem. We claim that the graph $H=G[V_c]$ that is induced by the central part $V_c$ of $(T,\cV)$ is as desired. 
Clearly, $H$ contains $U$. And $H$ is connected because the separators of $(T,\cV)$ are connected. Now if $R$ is any ray in $H$, it is dominated in $G$ by some vertex $v\in V_c$. This vertex $v$ also dominates 
$R$ in $H$ because every infinite $v$--$(R-v)$ fan in $G$ can be greedily turned into an infinite $v$--$(R-v)$ fan in $H$ by employing the connectedness of the finite separators of the star-decomposition.
\end{proof}

\bibliographystyle{amsplain}
\bibliography{StarCombBib}
\end{document}